\documentclass{article} %\documentclass[a4paper,12pt]{article}

\usepackage[utf8]{inputenc}
\usepackage{hyperref}
\usepackage{geometry}
\usepackage{pdfpages}
\usepackage{url}
\usepackage{amssymb}
\usepackage{amsthm}
\usepackage{amsmath}
\usepackage{setspace}
\usepackage{listings}
\usepackage[normalem]{ulem}
\usepackage{algorithmicx}
\usepackage[noend]{algpseudocode}
\usepackage{varwidth}
\usepackage{datetime}

\newtheorem{theorem}{Theorem}
\newtheorem{corollary}{Corollary}
\newtheorem{lemma}{Lemma}

\newcommand{\il}[1]{\text{\textbf{IL}\textsf{#1}}}

\newcommand{\sDeltapure}[0]{\Xi}
\newcommand{\sDeltacrt}[0]{\Delta}
\newcommand{\sDeltanp}[0]{\Delta^+}
\newcommand{\sDeltanm}[0]{\Delta^-}
\newcommand{\sGammai}[0]{\Gamma^\rhd}
\newcommand{\sGammast}[0]{\Gamma}
\newcommand{\sGammash}[0]{\Gamma_0}
\newcommand{\sGammapure}[0]{S}

\title{Complexity of the interpretability logic \il{}}
\author{Luka~Mikec\footnote{Department of Mathematics, Faculty of Science, University of Zagreb, Croatia\newline \hspace*{5.2mm}Supported by Croatian Science Foundation (HRZZ) under the project UIP-05-2017-9219.} \\
\texttt{luka.mikec@math.hr}
\and
Fedor~Pakhomov\footnote{Steklov Mathematical Institute of Russian Academy of Sciences, Moscow, Russia\newline \hspace*{5.2mm}National Research University Higher School of
Economics, Moscow, Russia \newline \hspace*{5.2mm}Supported in part by Young Russian Mathematics award}\\ 
\texttt{pakhfn@mi.ras.ru}
\and
Mladen~Vukovi\'{c}\footnote{Department of Mathematics, Faculty of Science, University of Zagreb, Croatia}\\
\texttt{vukovic@math.hr}}

\newcommand{\mydate}{\formatdate{2}{4}{2018}}
\date{\mydate}

\begin{document}
\maketitle

\begin{abstract}
\noindent
We show that the decision problem for the basic system of interpretability logic \il{} is \texttt{PSPACE}--complete. For this purpose we 
present an algorithm which uses polynomial space w.r.t.\ the complexity of a given formula. The existence of such algorithm, together with the previously known 
\texttt{PSPACE}--hardness of the closed fragment of \il{}, implies \texttt{PSPACE}--completeness. 
\end{abstract}


\section*{Introduction}\label{S:one}

Computational complexity of modal logics was first studied by Ladner \cite{ladner}.
Various tableau--based methods were used in proofs of \texttt{PSPACE}--decidability of a number of modal logics (like {\bf K}, {\bf K4}, 
{\bf S4} etc; see \cite{ladner} and \cite{spaan}).
\texttt{PSPACE}--completeness of the satisfiability problem (and also of the decision problem,
since co-\texttt{PSPACE} = \texttt{PSPACE}) for the closed fragments of modal systems {\bf K4}, {\bf S4}, {\bf Grz} and 
{\textbf {GL}} is proved by Chagrov and Rybakov  \cite{chagrov-rybakov}.
Shapirovsky \cite{shapirovsky} proved the \texttt{PSPACE}--decidability of propositional polymodal provability logic {\bf GLP}.
\texttt{PSPACE}--completeness of the closed fragment of the system {\bf GLP} is proved by Pakhomov in \cite{pakhomov}.

The interpretability logic \il{},  introduced by Visser \cite{Visser},  is an extension of provability logic with a binary modal  operator 
$\rhd.$ This operator stands for interpretability, considered as a relation between extensions of a fixed theory. In this paper we focus on 
modal aspects of interpretability logic. For details on arithmetical aspects see e.g.\ \cite{Visserover}.
Bou and Joosten proved in \cite{bou-joosten} that the decidability problem for the closed fragment of \il{} is \texttt{PSPACE}--hard.

We consider the complexity problem for interpretability logic and prove that the system \il{} is \texttt{PSPACE}--complete. 
Our constructions can be seen as generalizations of the constructions by Boolos presented in \cite{boolos} (Chapter 10). If we restrict our 
work to \textbf{GL}, the resulting method is very similiar to the one given by Boolos, up to the terminology. Our method can also be seen as extending the method presented in \cite{shapirovsky}, of proving 
\texttt{PSPACE}--completeness (monomodal case). 

\section{Preliminaries}

The language of interpretability logics is given by 
\[A::=p\,|\,\bot\,|\,A\to A\,|\,A\rhd A,\] 
where $p$ ranges over a fixed set of propositional variables. Other Boolean connectives can be defined as abbreviations as usual. 
Also, $\Box A$ can be defined as an abbreviation (it is provably equivalent to $\neg A\rhd\bot$), thus making this language an extension of 
the basic modal language. 
From this point on, by $\square A$ we will mean $\neg A\rhd\bot$, and by $\Diamond A$ we will mean $\neg(A\rhd\bot).$\\

Provability logic {\bf GL} is a modal logic with standard Kripke-style semantics (validity on transitive and reverse well--founded frames), 
which is sound and complete with respect to its interpretation in 
PA, where $\square$ is interpreted as a (formalized) provability predicate.
The axioms of {\bf GL} are all instances of the following schemata:
\begin{enumerate}
\item classical tautologies;
\item $\Box (A\rightarrow B)\rightarrow (\Box A\rightarrow \Box B)$;
\item $\Box (\Box A\rightarrow A)\rightarrow \Box A$.
\end{enumerate}
The rules of inference are:
\begin{enumerate}
\item $\displaystyle\frac{A\to B\;\;\; A}{B}$ (Modus ponens)
\item $\displaystyle\frac{A}{\Box A}$ (Necessitation)
\end{enumerate}  

Axioms of interpretability logic \il{} are all axioms of {\bf GL} and all the instances of the following schemata:
\begin{enumerate}
\item $\Box (A\rightarrow B)\rightarrow A\rhd B$;
\item $(A\rhd B)\ \wedge \ (B\rhd C)\rightarrow A\rhd C$;
\item $(A\rhd C) \ \wedge \ (B\rhd C) \rightarrow A\vee B\rhd C$;
\item $A\rhd B\rightarrow(\Diamond A \rightarrow \Diamond B)$;
\item $\Diamond A \rhd A$.
\end{enumerate}
As usual, we avoid parentheses if possible, treating $\rhd$ as having higher priority than $\rightarrow,$ 
but lower than other logical connectives. The inference rules are the same as for \textbf{GL}. 

The basic semantics for interpretability logic \il{} is provided by Veltman models. 
A \textit{Veltman frame} is a triple $\mathfrak{F}=(W,R, \{S_x : x\in W \})$, where $W$ is a non--empty set, $R$ is a transitive and 
reverse well--founded relation on $W$ (i.e.\ $(W,R)$ is a {\bf GL}--frame) and for all $x\in W$ we have:

\begin{quote}
\begin{itemize}
\item[a)] if $uS_xv$ then $xRu$ and $xRv$;
\item[b)] the relation $S_x$ is transitive and reflexive on $\{y \in W : xRy\}$; 
\item[c)] if $xRuRv$ then $uS_x v$.
\end{itemize}
\end{quote}

A  \textit{Veltman model}  is $\mathfrak{M}=(\mathfrak{F},\Vdash )$, where $\mathfrak{F}$ is a
Veltman frame and $\Vdash$ is a forcing relation, which is defined as usual in atomic and Boolean cases, and $x\Vdash A\rhd B$ if and 
only if for all $u$ such that $xRu$ and $u\Vdash A$ there exists $v$ such that $uS_x v$ and $v\Vdash B$.
De~Jongh and Veltman \cite{deJongh-Veltman} proved the completeness of the system \il{} w.r.t.\ finite Veltman models. Thus, 
 \il{} is decidable.

A \textit{rooted Veltman model} $(\mathfrak{M},w)$ is a pair consisting of a Veltman model $\mathfrak{M}=(W,R,\{S_x\mid x\in W\})$ and world 
$w$ such that all worlds are $R$--accessible from $w$; we say that $(\mathfrak{M},w)$ is a model of formula $\varphi$ (set of formulas $\Phi$) 
if $\mathfrak{M},w\Vdash \varphi$ ($\mathfrak{M},w\Vdash \varphi$, for each $\varphi\in \Phi$).

For a Veltman model $\mathcal{M}$ and world $x$, the \textit{rooted submodel generated by $x$} is the rooted model $(\mathcal{N},x)$, where 
$\mathcal{N}$ is the restriction of $\mathcal{M}$ to the set of all worlds that are either $x$ itself or are $R$--accessible from $x.$

We will say that $S \subseteq T$ is maximal Boolean consistent (w.r.t.\ $T$) if
$S$ is $\subseteq$--maximal propositionally consistent\footnote{Uniformly substitute all subformulas of the form $C \rhd D$ in all formulas 
from $S$ with fresh propositional variables. We say that $S$ is consistent if thus obtained set of propositional formulas is consistent.} set, 
i.e.\ there are no propositionally consistent sets $S' \subseteq T$ such that $S \subsetneq S'$.

\section{\texttt{PSPACE} algorithm}

We will present a  $\mathtt{PSPACE}$--algorithm that given an $\mathsf{IL}$-formula $\delta$ checks whether there is a rooted 
Veltman model $(\mathfrak{M},w)$ of $\delta.$

Let us denote by $Sub(\delta)$ the set of subformulas of a formula $\delta.$
For given formula $\delta$ we define the following sets of formulas:
$$\begin{array}{rcl}
\sGammash & = & \{ \varphi : \text{there is a formula } \psi \text{ such that } \varphi \rhd\psi \in Sub(\delta) 
\text{ or } \psi\rhd\varphi\in Sub(\delta)\}\\
\mbox{}\\
\sGammapure & = & Sub(\delta) \cup \{ \bot\} \cup \{ \varphi\rhd\bot : \varphi\in \sGammash\}\\
\mbox{}\\
\sGammast & = & \sGammapure\cup\{\lnot\varphi\mid \varphi\in \sGammapure\}\\
\mbox{}\\
\sGammai & = & \{\varphi\rhd \psi \mid \varphi\rhd \psi \in S\}
\end{array}
$$

In order to prove the theorem we will give a $\mathtt{PSPACE}$--algorithm (in $|\delta|$) (1) that checks for a given set 
$\sDeltapure\subset \sGammast$ whether there is a rooted Veltman model $(\mathcal{M},w)$ of  $\sDeltapure$ (in order to give an answer for our 
formula $\delta$ we will apply it to the set $\{\delta\}$).
  
We will give our algorithm as main function (1) and supplementary functions (2) and (3) that could make recursive calls of each other and 
return either positive or negative answer ((1) makes only calls of (2), (2) makes only calls of (3), and (3) makes only calls of (1)). 
First we will give full description of computation process and specify what we are computing, but we will prove our claims about what we are 
computing only latter.

In order to compute (1) on a given input $\sDeltapure$, we consider each of its maximal Boolean consistent extensions 
$\sDeltacrt\subset \sGammast$  and make checks (2) of  whether there are rooted Veltman models $(\mathcal{M},w)$ of $\sDeltacrt$; we return 
positive result for (1) iff at least one of the checks return positive answer.

Now we describe how we make check (2). We consider sets 
\begin{equation}\label{Delta+-def}\sDeltanp=\sDeltacrt\cap\sGammai \ \mbox{ and } \ \sDeltanm=\{\varphi\mid \lnot \varphi\in \sDeltacrt\}\cap 
\sGammai.\end{equation}  
For each formula $\zeta\in \sDeltanm$ ($\zeta$ is of the form 
$\chi\rhd\eta$) we make the following check (3)  of whether there is a rooted Veltman model 
$(\mathcal{H}_{\zeta},h_{\zeta})$ of $\sDeltanp$, where $\zeta$ fails; we return positive answer for (2) iff all the answers for (3) are 
positive.

\vskip 2ex
We will say that a pair $(\Sigma, \Theta)$ is a $(\sDeltacrt,\zeta)$-pair if:
\begin{itemize}
\item $\Sigma, \Theta \subseteq\sGammash$;
\item $\eta\in \Sigma$, $\bot\not\in \Theta$;
\item for each $\varphi\rhd\psi \in \sDeltanp$, either $\varphi\in \Sigma$ or $\psi\in \Theta$.
\end{itemize}

In order to make check (3), we return positive answer if there is a  $(\sDeltacrt,\zeta)$-pair $(\Sigma, \Theta)$ such that all of the 
following holds (we make calls of (1) to check conditions below):
\begin{enumerate}
\item \label{c1} there is a rooted Veltman model $(\mathcal{E},e)$ of 
$\{\lnot \sigma,\sigma\rhd\bot\mid \sigma\in \Sigma\}\cup \{\chi,\chi\rhd\bot\}$;

\item \label{c2} for each $\theta\in \Theta$ there is a rooted Veltman model $(\mathcal{G}_{\theta},g_{\theta})$ of 
$\{\lnot \sigma,\sigma\rhd\bot\mid \sigma\in \Sigma\}\cup \{\theta,\theta\rhd\bot\}$.
\end{enumerate}

We will now proceed to prove that the algorithm above has the required properties. To do so, let us first verify that (1), (2) and (3) do what 
we described. 

First, the following lemma is obvious:
\begin{lemma}
\label{algcor1}
The following are equivalent: 
\begin{enumerate}
\item there exists a rooted Veltman model $(\mathcal{M},w)$ of $\sDeltapure\subseteq \sGammast$;
\item there exists a rooted Veltman model of some maximal Boolean consistent extension $\sDeltacrt\subseteq \sGammast$ of $\sDeltapure$.
\end{enumerate}
\end{lemma}

\begin{lemma}
\label{algcor2}
Let $\sDeltacrt$ be a maximally propositionally consistent subset of $\sGammast$. The sets $\sDeltanp$ and $\sDeltanm$ are given by 
(\ref{Delta+-def}). The following are equivalent: 
\begin{enumerate}
\item there exists a rooted Veltman model of a maximal Boolean consistent set $\sDeltacrt\subseteq \sGammast$;
\item for all $\zeta\in \sDeltanm$, there is a rooted Veltman model 
$(\mathcal{H}_{\zeta},h_{\zeta})$ of $\sDeltanp$, where $\zeta$ fails.
\end{enumerate}
\end{lemma}
\begin{proof}
First assume that indeed there is a rooted Veltman model $(\mathcal{M}, w)$ of $\sDeltacrt$. 
It is easy to see that we could just put $(\mathcal{H}_{\zeta}, h_{\zeta})=(\mathcal{M}, w)$, for each $\zeta\in\sDeltanm$. 

In the other 
direction, suppose we have rooted Veltmam models with described properties $(\mathcal{H}_{\zeta}, h_{\zeta})$ for each $\zeta\in\sDeltanm$. 
In order to construct rooted Veltman model $(\mathcal{M},w)$ of $\sDeltacrt$, we take disjoint union all the models $\mathcal{H}_{\zeta}$, 
then merge worlds $h_{\zeta}$ in one world $w$ and put satisfaction of proposition variables in $w$ according to set $\sDeltacrt$.  

It is easy to prove by induction on the complexity of a formula $\varphi\in\Gamma$ that we have the following: $\mathcal{M},w\Vdash\varphi$
if and only if $\varphi\in\Delta.$ We consider only the case $\varphi=\psi_1\rhd \psi_2.$
Suppose $\mathcal{M},w\Vdash\varphi$. If $\varphi \not\in \Delta$ then $\neg \varphi \in \Delta$, and so 
$\varphi \in \Delta^-$. By assumption, $\varphi$ fails in $\mathcal{H}_{\varphi}$, so there is some $x$, $h_\varphi R_\varphi x$, with 
$\mathcal{H}_\varphi, x \Vdash \psi_1$ and for no $y$, $x S_{h_\varphi} y$, $\mathcal{H}_\varphi, y \Vdash \psi_2$. Since $x$ is included in 
$\mathcal{M}$ and $wRx$, this contradicts the assumption that $\mathcal{M},w\Vdash\varphi$. 
In the other direction, suppose $\varphi \in \Delta$, and $wRx$. Since $\varphi \in \Delta$, also $\varphi \in \Delta^+$. Since $wRx$, by 
construction, $x$ is in $\mathcal{H}_{\varphi'}$ for exactly one $\varphi'$. Since $\mathcal{H}_{\varphi'}$ is a model of $\Delta^+$, there is 
$y$ such that $x S_{h_{\varphi'}} y$ and $\mathcal{H}_{\varphi'},y \Vdash \psi_2$. But $y$ is included in $\mathcal{M}$, and we have $x S_{w} y$.
\end{proof}

\begin{lemma}
\label{algcor3}
Let $\sDeltacrt\subseteq\sGammast$ be a maximal propositionally consistent set, $\sDeltanp$ and $\sDeltanm$ -
be given by (\ref{Delta+-def}), 
and $\zeta=\chi\rhd\eta$ be a formula from $\sGammast$. The following are equivalent: 
\begin{enumerate}
\item there is a rooted Veltman model $(\mathcal{H}_{\zeta},h_{\zeta})$ of $\sDeltanp$, where $\zeta$ fails;
\item there is a $(\sDeltacrt,\zeta)$-pair $(\Sigma, \Theta)$ such that all of the following holds:
    \begin{enumerate}
    \item \label{c1} there is a rooted Veltman model $(\mathcal{E},e)$ of 
    $\{\lnot \sigma,\sigma\rhd\bot\mid \sigma\in \Sigma\}\cup \{\chi,\chi\rhd\bot\}$;
    \item \label{c2} for each $\theta\in \Theta$ there is a rooted Veltman model $(\mathcal{G}_{\theta},g_{\theta})$ of 
    $\{\lnot \sigma,\sigma\rhd\bot\mid \sigma\in \Sigma\}\cup \{\theta,\theta\rhd\bot\}$.
    \end{enumerate}
\end{enumerate}
\end{lemma}
\begin{proof}
First let us assume that we found a pair $(\Sigma,\Theta)$ with required properties and then show that there exists a rooted Veltman model 
$(\mathcal{H}_{\zeta},h_{\zeta})$ of $\sDeltanp$ in which $\chi\rhd\eta$ fails. We have rooted Veltman model $(\mathcal{E},e)$ of 
$\{\lnot \sigma,\sigma\rhd\bot\mid \sigma\in \Sigma\}\cup \{\chi,\chi\rhd\bot\}$ and rooted Veltman models $(\mathcal{G}_{\theta},g_{\theta})$ 
of $\{\lnot \sigma,\sigma\rhd\bot\mid \sigma\in \Sigma\}\cup \{\theta,\theta\rhd\bot\}$, for $\theta\in \Theta$. Now we consider the disjoint 
union of $\mathcal{E}$ and all $\mathcal{G}_{\theta}$ then add to the union new world $h_{\zeta}$, make all other worlds $R$-accessible from 
$h_{\zeta}$, and  make all the worlds other than $h_{\zeta}$ pairwise accessible by $S_{h_{\zeta}}$ (satisfaction of variables in $h_{\zeta}$ 
could be arbitrary). We take the resulting model as the desired model $\mathcal{H}_{\zeta}$. It is easy to see that all formulas from $\Sigma$ 
fail in every world of $\mathcal{H}_{\zeta}$ different from $h_{\zeta}$. Thus $\mathcal{H}_{\zeta},h_{\zeta}\Vdash\sigma\rhd\psi$, for all 
$\sigma$ from $\Sigma$. And since $\eta\in \Sigma$ and $\mathcal{H}_{\zeta},e\Vdash \chi$, the formula 
$\mathcal{H}_{\zeta},h_{\zeta}\nVdash \chi\rhd\eta$. Since for each $\theta\in \Theta$, the world $g_{\theta}$ is $S_{h_{\zeta}}$-accessible 
or each $\theta\in \Theta$ and arbitrary $\varphi$ we have $\mathcal{H}_{\zeta},h_{\zeta}\Vdash \varphi\rhd\theta$. Finally we conclude that 
$(\mathcal{H}_{\zeta},h_{\zeta})$ is indeed a model of $\sDeltanp$ in which $\chi\rhd\eta$ fails.

Now assume that we have a rooted Veltman model $(\mathcal{H}_{\zeta},h_{\zeta})$ of $\sDeltanp$ in which $\chi\rhd\eta$ 
fails. We now need to find pair of sets $(\Sigma,\Theta)$, models $(\mathcal{E},e)$, and models $(\mathcal{G}_{\theta},g_{\theta})$, for 
$\theta\in \Theta$  with all the desired properties from (3). Let $e$ be some $R$--maximal world of $\mathcal{H}_{\zeta}$ in which $\chi$ 
holds but for which there are no $S_{h_{\zeta}}$-accessible world, where $\eta$ holds. We consider the set $A$ of all $\mathcal{H}_{\zeta}$-
worlds that are $S_{h_{\zeta}}$-accessible from $e$. We take $\Sigma$ to be the set of all formulas from $\sGammash$ that fail in all the 
worlds from $A$ and $\Theta$ to be the set of all formulas from $\sGammash$ that hold at least in one world from $A$.  Now it is easy to see 
that we could take the rooted submodel  of $\mathcal{H}_{\zeta}$ generated by $e$ as $(\mathcal{E},e)$ (it is easy to see that it is a model 
of $\{\lnot \sigma,\sigma\rhd\bot\mid \sigma\in \Sigma\}\cup \{\chi,\chi\rhd\bot\}$). For each 
$\theta\in \Theta$ we put $g_\theta$ to be some $R$--maximal element of $A$ such that $\mathcal{H}_{\zeta},g_{\theta}\Vdash\theta$. We take as 
models $(\mathcal{G}_{\theta},g_{\theta})$ the rooted submodels of $(\mathcal{H}_{\zeta},h_{\zeta})$ generated by respective $g_{\theta}$. Due 
to the definition of $\Sigma$, all $(\mathcal{G}_{\theta},g_{\theta})$ are models of 
$\{\lnot \sigma,\sigma\rhd\bot\mid \sigma\in \Sigma\}$
and due to the choice of $g_{\theta}$'s we have 
$(\mathcal{G}_{\theta},g_{\theta})\Vdash \theta$ and $(\mathcal{G}_{\theta},g_{\theta})\Vdash \theta\rhd\bot$, for each 
$\theta\in \Theta$.
\end{proof}

\begin{theorem} Logic $\mathsf{IL}$ is $\mathtt{PSPACE}$--decidable. \end{theorem}
\begin{proof}
First we show that the depth of recursion calls is bounded by $|\sGammai|+2$. Indeed, we notice that if some formula 
$\varphi\rhd\bot$ is in $\sDeltanp$ from some execution of (2) then $\varphi\rhd\bot$  will be in  
$\sDeltanp$ of all the deeper calls of (2). We will show that in each deeper call of (2) the new $\sDeltanp$ will contain at least one new 
formula of the form $\varphi\rhd\bot$; clearly this will give us bound $|\sGammai|+1$ on depth of recursion 
calls of (2) and hence the desired bound $|\sGammai|+2$ on depth of recursion 
calls of (1). Let us consider an execution of (3) and show that all the deeper calls of (2) that will be made from this point will contain 
some additional formula of the form $\varphi\rhd \bot$ in (new) $\sDeltanp$. In this execution of (3) all the further recursive calls are made 
either from check 1. or from check 2. for some $\theta\in \Theta$; note that each check here makes exactly one call of (1). Since the cases of 
checks 1. and 2. are very similar, we will consider only the call of (1) made from check 1. It is enough to show that if the formula 
$\chi\rhd\bot$ is already in $\sDeltanp$ then there will no deeper calls of (2) in this branch of computation. Indeed, in this case $\chi$ is 
in $\Sigma$ (otherwise $\bot$ should have been in 
$\Theta$ which is forbidden), hence the recursive call of (1) that will be made at the point will be with propositionally inconsistent set as 
an argument and thus will not make further calls of (2). 

Since each individual procedure (without recursive calls) in our computation clearly is $\mathsf{PSPACE}$ in $|\delta|$, and the depth of 
calls have linear in $|\delta|$ upper bound, the whole procedure (1) (accounting recursive calls) is $\mathsf{PSPACE}$ and of course 
terminates.

Now the only thing left for us to verify is that the descriptions of what we are computing in our algorithm are indeed correct. Since we 
already know that our computation terminates, it is enough to show that our algorithm works locally correct, i.e. that assuming that further 
calls do what we describe that they are doing, the call under consideration also computes what we want it to compute. Formally, we prove the 
correctness of descriptions by induction on depth of recursive calls (the base case are terminal calls, i.e.\ leaf calls in the execution 
tree).

The fact that the description of (1) is correct follows from Lemma \ref{algcor1}, that the description of (1) is correct follows from Lemma 
\ref{algcor2}, and that the description of (1) is correct follows from Lemma \ref{algcor3}.

Thus, there is a model of $\phi$ if and only if (1) returns a positive answer given $\{\phi\}$ as an input.
\end{proof}

Bou and Joosten \cite{bou-joosten} proved that the decidability problem for the closed fragment of \il{} is \texttt{PSPACE}--hard. Together 
with the previous theorem, this implies the following.

\begin{corollary}
The decidability problem of the logic \il{} is  \texttt{PSPACE}--complete.
\end{corollary}

It is natural to ask whether this result extends to other interpretability logics. Perhaps the best candidates for the future research are 
interpretability logics that are known to be decidable. These are (to the best of our knowledge) \il{P} 
(\cite{deJongh-Veltman}), \il{M} (\cite{deJongh-Veltman}), \il{W} (\cite{deJongh-Veltman-1998}), 
\il{M$_0$} (\cite{nas-clanak}) and \il{W$^*$} (\cite{nas-clanak}). 

Note also that in \cite{nas-clanak} the decidability of certain logics was proved using generalized Veltman semantics, in which $S_w$-
successors are sets of 
worlds. Therefore an adaptation of the technique of this paper should take that into consideration.

\end{document}